\documentclass[11pt,a4paper]{amsart}
\usepackage[utf8]{inputenc}
\usepackage[T1]{fontenc}
\usepackage{amsmath, amssymb, amsthm, amsfonts}
\usepackage{geometry}
\usepackage{enumitem}

\newtheorem{theorem}{Theorem}[section]
\newtheorem{proposition}[theorem]{Proposition}
\newtheorem{lemma}[theorem]{Lemma}
\newtheorem{definition}[theorem]{Definition}
\newtheorem{corollary}[theorem]{Corollary}
\newtheorem{remark}[theorem]{Remark}
\newtheorem{example}[theorem]{Example}
\newtheorem{notation}[theorem]{Notation}

\newcommand{\JT}{JT}
\newcommand{\Tree}{\mathcal{T}}
\newcommand{\norm}[1]{\left\lVert#1\right\rVert}
\newcommand{\supp}{\operatorname{supp}}
\newcommand{\ran}{\operatorname{ran}}
\newcommand{\J}{J}
\newcommand{\N}{N}
\newcommand{\Branches}{\mathcal{B}}

\begin{document}

\title{A Study of the Extreme Points in the Unit Ball of $\JT$}
\author[Spiros A. Argyros]{Spiros A. Argyros}
\address{National Technical University of Athens, Faculty of Applied Sciences, 
Department of Mathematics, Zografou Campus, 157 80, Athens, Greece.}
\email{sargyros@math.ntua.gr}
\date{\today}

\begin{abstract}
In this note, we investigate the extreme points of the unit ball of the James Tree space ($\JT$). We relate the geometric structure of $\JT$ to the classical James space $\J$ and provide partial characterizations of extremality based on the concept of separated vectors. We provide a complete characterization for positive vectors and establish the equal sums property for complete positive extreme points.
\end{abstract}

\makeatletter
\@namedef{subjclassname@2020}{\textup{2020} Mathematics Subject Classification}
\makeatother


\maketitle

\section{Introduction}

The motivation for this research stems from the characterization of the extreme points of the unit ball of the classical James space $\J$, proved by S. Bellenot \cite{B} and recently revisited by S. Argyros and M. González \cite{AG2025}. In that context, it was established that a vector in $\J$ is an extreme point if and only if it satisfies a geometric property termed ``separated'' (or equivalently, if its norm coincides with its $\ell_2$ norm). This result naturally raises the question whether the extreme points of the unit ball of $\JT$ satisfy similar properties, as $\JT$ is the extension of $\J$ on the dyadic tree. 

In this note, we investigate the problem for the James Tree space $\JT$. The space $\JT$ is defined as the completion of the space of finitely supported functions on the dyadic tree $\Tree$ under the norm:
\[ \norm{x}_{\JT} = \sup_{\mathcal{P}} \left( \sum_{S \in \mathcal{P}} \left( \sum_{\alpha \in S} x(\alpha) \right)^2 \right)^{1/2} \]
where the supremum is taken over all families $\mathcal{P}$ of pairwise disjoint segments in $\Tree$. Such families $\mathcal{P}$ will be called partitions. A partition $\mathcal{P}$ for which this supremum is attained is called an \textbf{$x$-norming partition}.

Central to our study is the concept of a \textbf{separated vector}. A vector $x \in \JT$ is said to be separated if for every pair of distinct nodes $\alpha, \beta$ in its range, there exists an $x$-norming partition that separates them (i.e., $\alpha$ and $\beta$ belong to different segments).

The main problem we study here is whether the separation of a vector $x \in \JT$ implies that $x$ is an extreme point of the ball of radius $\|x\|$.

We recall that for the classical James space $\J$, the condition for extremality is stringent: $x$ is an extreme point if and only if $\norm{x}_{\J} = \norm{x}_2$. In $\JT$, this equivalence holds for vectors supported on a single branch, but the tree structure allows for more complex extremal behaviors.

We provide an affirmative answer to the main problem for vectors with finite support. For the cone of positive vectors, we provide a complete answer: a positive vector is an extreme point if and only if it is separated. Unlike the case of the space $\J$, where the only positive extreme points are the elements of the basis, in $\JT$ there are many positive extreme points. For example, any vector $x_n= \sum_{|\alpha | \le n} e_{\alpha}$ is an extreme point of the ball of radius $\norm{x_n}$. The analysis of positive vectors is significantly facilitated by the existence of a ``Greedy Algorithm'', which we introduce to construct $x$-norming partitions effectively.

Finally, for a class of positive extreme points, we prove the validity of the \textbf{Equal Sums Property}. This property states that for any node in the support, the sum of the vector along any branch segment descending from that node is constant. The positive extreme points that satisfy this property are those with complete support. This result highlights the rigid structure imposed by extremality on complete positive vectors.

\section{Preliminaries}

In this section, we establish the notation and basic properties of the dyadic tree and define the space $\JT$. We also introduce the geometric concepts of extreme points and separated vectors.

\subsection{The Dyadic Tree Structure}

Let $\Tree$ denote the dyadic tree. The elements of $\Tree$ are finite sequences of $0$'s and $1$'s, including the empty sequence $\emptyset$ (the root). We equip $\Tree$ with the natural partial order: $\alpha \le \beta$ if $\alpha$ is an initial segment of $\beta$.

\begin{definition}[Segments]
    A subset $S \subset \Tree$ is called a \textbf{segment} if it is convex with respect to the tree order and linearly ordered. That is, for any $\alpha, \beta \in S$ with $\alpha \le \beta$, if $\gamma \in \Tree$ satisfies $\alpha \le \gamma \le \beta$, then $\gamma \in S$. A \textbf{branch} of $\Tree$ is a maximal segment.
\end{definition}

\begin{definition}[Complete Subtree]
  A $\textbf{subtree}$ of $\Tree$ is any non empty subset of $\Tree$ with the partial order induced by that tree.  A subset $K \subseteq \Tree$ is called a \textbf{complete subtree} if for any $\alpha, \beta \in K$ with $\alpha \le \beta$, the entire segment $[\alpha, \beta] = \{ \gamma \in \Tree : \alpha \le \gamma \le \beta \}$ is contained in $K$.
\end{definition}

\subsection{The Space $\JT$}

The James Tree space, denoted by $\JT$, is constructed as a completion of finitely supported functions on $\Tree$.

\begin{definition}[Support and Range]
    Let $x: \Tree \to \mathbb{R}$ be a function.
    \begin{enumerate}
        \item The \textbf{support} of $x$, denoted $\supp(x)$, is the set of nodes where $x$ is non-zero:
        \[ \supp(x) = \{ \alpha \in \Tree : x(\alpha) \neq 0 \} \]
        \item The \textbf{range} of $x$, denoted $\ran(x)$, is defined as the smallest complete subtree of $\Tree$ that contains $\supp(x)$.
    \end{enumerate}
\end{definition}

\begin{definition}[The Norm of $\JT$]
    Let $c_{00}(\Tree)$ be the space of functions $x: \Tree \to \mathbb{R}$ with finite support. For $x \in c_{00}(\Tree)$, we define the norm:
    \begin{equation}
        \norm{x}_{\JT} = \sup_{\mathcal{P}} \left( \sum_{S \in \mathcal{P}} \left( \sum_{\alpha \in S} x(\alpha) \right)^2 \right)^{1/2}
    \end{equation}
    where the supremum is taken over all families $\mathcal{P}$ of pairwise disjoint segments in $\Tree$. The space $\JT$ is the completion of $(c_{00}(\Tree), \norm{\cdot}_{\JT})$.
\end{definition}

Throughout this note a partition $\mathcal{P}$ is a collection of pairwise disjoint segments of $\Tree$. For a fixed partition $\mathcal{P} = \{S_i\}_{i}$, we denote the value of $x$ with respect to this partition as:
\[ \norm{x}_{\mathcal{P}} = \left( \sum_{S \in \mathcal{P}} \left( \sum_{\alpha \in S} x(\alpha) \right)^2 \right)^{1/2} \]
For $x\in \JT$, a partition $\mathcal{P}$ is called an \textbf{$x$-norming partition} if 
\[ \norm{x}_\JT = \norm{x}_\mathcal{P}. \]
\begin{lemma} \label{L10} Let $x\in \JT$ and $S$ a segment of $\Tree$. Assume that $S\cap \supp(x) = \emptyset$ and $S \not\subset \ran(x) $.
 Then  $S\setminus \ran(x) $ is a segment of $\Tree$.
 \end{lemma}
 \begin{proof}  It is easy to check that $\ran(x)$ has the following description.
   \[ \ran(x) = \left\{ c: \exists \,u, v \in \supp(x) \quad \text{such that} \quad u \le c \le v \right\} \]
   Observe that the previous implies that every $u \in \ran(x)$ which is either maximal or minimal element of the subtree $\ran(x)$ belongs to $\supp(x).$
   Assume that $S\cap \ran(x) \ne \emptyset$ (otherwise the conclusion is trivial).   Further assume, on the contrary, that 
   $S\setminus \ran(x)$ is not a segment of $\Tree$. We set $I = S\cap \ran(x)$ Since $\min S < \min I$ it follows that
   $ c = \min I$ is a minimal element of the subtree $\ran(x)$ which belongs to $\supp(x)$ a contradiction.
   
   \end{proof}

In the following, the first part is well known.

\begin{lemma}\label{lem1}
    (1) Let $\mathcal{B}$ be a complete subtree of $\Tree$. Then for every $x \in \JT$ we have 
    \[ \norm{x}_{\JT} \ge \norm{x|_{\mathcal{B}}}_{\JT}. \]
    (2) For $x\in \JT$, let $\mathcal{P}$ be an $x$-norming partition and $\mathcal{B}$ a complete subtree. We set
     \[  \mathcal{P}' =\left\{ S\in \mathcal{P}: S \subset \mathcal{B} \right\} \]
     and we assume that \[\supp(y)\subset \bigcup_{S\in \mathcal{P}'} S \quad \text{where}\quad y= x|_\mathcal{B}.\]
     
     Then $\mathcal{P}'$ is a $y$-norming partition for $y = x|_{\mathcal{B}}$.
    
\end{lemma}

\begin{proof}
    (1) Let $y = x|_{\mathcal{B}}$. Let $\mathcal{P}$ be any partition of $\Tree$ (i.e., a family of pairwise disjoint segments of $\Tree$). Then the family 
    \[ \hat{\mathcal{P}} = \{S \cap \mathcal{B} : S \in \mathcal{P} \text{ and } S \cap \mathcal{B} \neq \emptyset\} \]
    is also a partition of $\Tree$ (since the intersection of a segment with a complete subtree is again a segment). 
    
    Furthermore, for every $S \in \mathcal{P}$, since $y$ vanishes outside of $\mathcal{B}$, it holds that:
    \[ \sum_{\alpha \in S} y(\alpha) = \sum_{\alpha \in S \cap \mathcal{B}} x(\alpha). \]
    Hence, $\norm{y}_{\mathcal{P}} = \norm{x}_{\hat{\mathcal{P}}} \le \norm{x}_{\JT}$. 
    
    Therefore, taking the supremum over all such partitions $\mathcal{P}$, we obtain:
    \[ \norm{y}_{\JT} = \sup \left\{ \norm{y}_{\mathcal{P}} : \mathcal{P} \text{ is a partition of } \Tree \right\} \le \norm{x}_{\JT}, \]
    which completes the proof of the first part.
    
    (2) Assume on the contrary that there exists a partition $\mathcal{Q}$ such that $\norm{y}_\mathcal{Q} > \norm{y}_{\mathcal{P}'}$. Since $\ran(y)$ is a complete subtree we may assume that the segments $S \in \mathcal{Q}$ are subsets of $\ran(y)$. Further for $S  \in \mathcal{P} \: \text{with} \: S\not\subset \mathcal{B} $ we set $ S' = S\setminus \ran(y) $, which from Lemma \ref{L10}, is a segment and we observe that from our assumptions we have 
    \[ \sum_S x(\alpha) = \sum_{S'} x(\alpha) \]
    
   We set \[\mathcal{P}'_2 = \left\{ S' = S\setminus \ran(x) : S \in \mathcal{P} \quad \text{and} \quad S \not\subset \mathcal{B}\right\}\]
    Then it is easy to see that for the partition  
     \[ \mathcal{Q}' = \mathcal{Q} \cup  \mathcal{P}'_2 \]
     we have $\norm{x}_\mathcal{P} < \norm{x}_{\mathcal{Q}'}$, a contradiction completing the proof.
\end{proof}

\begin{proposition}\label{P1}[Existence of Norming Partitions]
  (i)   For every $x \in \JT$, there exists at least one partition $\mathcal{P}$ such that the norm is attained:
    \[ \norm{x}_{\JT} = \norm{x}_{\mathcal{P}}. \]
    (ii) Let $x\in\JT$ and $\{\mathcal{P}_n\}_n$ a sequence of partitions such that 
    \[ \norm{x}_{\mathcal{P}_n}> \norm{x}_{\JT} -\epsilon_n \quad \text{with} \quad  \epsilon_n > 0, \: \epsilon_n \to 0. \]
    Then there exists a subsequence $\{\mathcal{P}_{n_k} \}_k$ and a partition $\mathcal{P}$ such that 
    \[ \mathcal{P}_{n_k} \xrightarrow{\text{pw}} \mathcal{P} \quad \text{and} \quad \norm{x}_\mathcal{P} = \norm{x}_{\JT}. \]

\end{proposition}

\begin{proof}
   (i)  We use a compactness argument. Let $\Pi$ be the set of all admissible partitions (collections of pairwise disjoint segments). We can encode any partition $\mathcal{P} \in \Pi$ as an element of the product space $K = \{0, 1\}^{\Tree} \times \{0, 1\}^{\Tree}$ by defining two indicator functions for each node $\alpha \in \Tree$:
    \begin{itemize}
        \item $v(\alpha) = 1$ if $\alpha$ belongs to some segment in $\mathcal{P}$, and $0$ otherwise.
        \item $e(\alpha) = 1$ if $\alpha$ belongs to the same segment as its parent, and $0$ otherwise (for the root $\emptyset$, $e(\emptyset) = 0$).
    \end{itemize}
    For a partition $\mathcal{P} = \{ I_i\}_{i\in F}$ the encoding $v$ creates the indicator function of $\bigcup_{i\in F}I_i$ and 
    the encoding $e$ the indicator function of $ \bigcup_{i\in F} (I_i\setminus\{\min I_i\})$. It is easy to check that the encoding $v, e$ determine uniquely each partition $\mathcal{P}$.

    The configurations in $K$ that correspond to valid partitions are exactly those satisfying two local conditions for all $\alpha \in \Tree$:
    \begin{enumerate}
        \item If $e(\alpha) = 1$, then $v(\alpha) = 1$ and $v(\text{parent}(\alpha)) = 1$.
        \item $\sum_{\beta \in \operatorname{children}(\alpha)} e(\beta) \le 1$ (each node connects to at most one child in the same segment).
    \end{enumerate}
    Since these constraints are strictly local (involving only a node, its parent, and its children), $\Pi$ is a closed subset of the compact space $K$ (by Tychonoff's theorem), making $\Pi$ itself a compact Hausdorff space under the product topology.

    For any vector $z \in \JT$, let $f_z: \Pi \to \mathbb{R}$ be the functional defined by $f_z(\mathcal{P}) = \norm{z}_{\mathcal{P}}$. 
    If $z \in c_{00}(\Tree)$ (i.e., $z$ has finite support), $f_z$ depends only on the coordinates $(v(\alpha), e(\alpha))$ for $\alpha$ in the minimal complete subtree containing the finite support of $z$ which is also a finite set. Thus, $f_z$ is a continuous function on $\Pi$.

    Now, let $x \in \JT$. Since $c_{00}(\Tree)$ is dense in $\JT$, there exists a sequence $(x_n)_{n=1}^\infty$ in $c_{00}(\Tree)$ such that $\norm{x - x_n}_{\JT} \to 0$ as $n \to \infty$.
    For a fixed partition, $\norm{\cdot}_{\mathcal{P}}$ is a semi-norm, so by the triangle inequality we have:
    \[ |f_x(\mathcal{P}) - f_{x_n}(\mathcal{P})| = \left| \norm{x}_{\mathcal{P}} - \norm{x_n}_{\mathcal{P}} \right| \le \norm{x - x_n}_{\mathcal{P}} \le \norm{x - x_n}_{\JT} \]
    for all $\mathcal{P} \in \Pi$. This shows that the sequence of continuous functions $f_{x_n}$ converges uniformly to $f_x$ on $\Pi$.
    
    As a uniform limit of continuous functions, $f_x: \Pi \to \mathbb{R}$ is continuous. Since $\Pi$ is compact, $f_x$ attains its maximum. Therefore, there exists a partition $\mathcal{P}^* \in \Pi$ such that:
    \[ \norm{x}_{\mathcal{P}^*} = \max_{\mathcal{P} \in \Pi} \norm{x}_{\mathcal{P}} = \norm{x}_{\JT}. \]
    This $\mathcal{P}^*$ is the required $x$-norming partition.
    
    (ii) For each $n\in \mathbb{N}$ consider the coding $(v(\mathcal{P}_n),  e( \mathcal{P}_n )) \in \Pi$. From the compactness of $\Pi$ there exists a subsequence 
    \[  \{ (v (\mathcal{P}_{n_k}), e(\mathcal{P}_{n_k} )) \}_k  \to  ( v (\mathcal{P}), e(\mathcal{P})) \quad \text{with}  
    \quad \mathcal{P} \: \text{partition}. \]
  Clearly $ \mathcal{P}_{n_k} \xrightarrow{\text{pw}} \mathcal{P}$ and the continuity of the function $f_x$ together with our assumptions yield that 
    \[ \norm{x}_\mathcal{P} = \norm{x}_{\JT}. \]
\end{proof}
\begin{remark} \label{Re1} Notice that the above proof yields that $\mathcal{P}_{n_k}  \xrightarrow{\text{pw}} \mathcal{P}$ in the following sense. If $s\in \mathcal{P}$ is a segment there exists a sequence of segments $(s_{n_k})_k$, $s_{n_k} \in \mathcal{P}_{n_k}$ such that $\chi_{s_{n_k}} \xrightarrow{\text{pw}} \chi_s$.
\end{remark}

\subsection{Extreme Points and Separated Vectors}

We focus on the geometric structure of the unit sphere of $\JT$.

\begin{definition}[Extreme Point]
    A non-zero vector $x \in \JT$ is called an \textbf{extreme point} of the ball of radius $\norm{x}$ if it cannot be written as a non-trivial convex combination of two distinct vectors with the same norm. Formally, if $x = \lambda y + (1-\lambda)z$ with $0 < \lambda < 1$ and $\norm{y}_{\JT} = \norm{z}_{\JT} = \norm{x}_{\JT}$, then $x = y = z$.
\end{definition}

A crucial geometric property in $\JT$ is the concept of separation. As we will see next, separation within the support is insufficient for extremality; we must consider the entire range.

\begin{definition}[Separated Vector]
    A vector $x \in \JT$ is said to be \textbf{separated} if for every pair of distinct nodes $\alpha, \beta \in \ran(x)$, there exists an $x$-norming partition $\mathcal{P}$ such that $\alpha$ and $\beta$ belong to different segments in $\mathcal{P}$.
\end{definition}

In the sequel, by the term that $x$ is an extreme point, we will mean an extreme point of the ball with radius $\norm{x}$. The following proposition provides a specific characterization of non-extremality in $\JT$, relating perturbations to the norming partitions.

\begin{proposition}[Characterization of Extreme Points]\label{prop.1}
    A vector $x \in \JT$ is \textbf{not} an extreme point if and only if there exists a non-zero vector $y \in \JT$ such that:
    \begin{enumerate}
        \item $\norm{x}_{\JT} = \norm{x+y}_{\JT} = \norm{x-y}_{\JT}$, and
        \item For every $x$-norming partition $\mathcal{P} = \{S_i\}$, we have $\norm{y}_{\mathcal{P}} = 0$.
    \end{enumerate}
    Note that $\norm{y}_{\mathcal{P}} = 0$ implies that for every segment $S \in \mathcal{P}$, $\sum_{\alpha \in S} y(\alpha) = 0$.
\end{proposition}

\begin{proof}
    ($\Leftarrow$) Suppose such a vector $y \neq 0$ exists. We can write $x = \frac{1}{2}(x+y) + \frac{1}{2}(x-y)$. Since $\norm{x+y}_{\JT} = \norm{x-y}_{\JT} = \norm{x}_{\JT}$, $x$ is the midpoint of a non-trivial segment on the sphere. Thus, $x$ is not an extreme point.

    ($\Rightarrow$) Suppose $x$ is not an extreme point. Then there exist distinct vectors $u, v \in \JT$ with $\norm{u} = \norm{v}_{\JT} = \norm{x}_{\JT}$ such that $x = \frac{1}{2}(u+v)$. Let $y = \frac{1}{2}(u-v) \neq 0$. Then $u = x+y$ and $v = x-y$, satisfying condition (1).

    Now, let $\mathcal{P} = \{S_i\}$ be any $x$-norming partition. For any segment $S \in \mathcal{P}$, let $x_S = \sum_{\alpha \in S} x(\alpha)$ and $y_S = \sum_{\alpha \in S} y(\alpha)$. By definition,
    \[ \norm{x}_{\JT}^2 = \sum_{S \in \mathcal{P}} x_S^2. \]
    Since $\norm{x \pm y}_{\JT} \ge \norm{x \pm y}_{\mathcal{P}}$, we have:
    \[ \norm{x \pm y}_{\JT}^2 \ge \sum_{S \in \mathcal{P}} (x_S \pm y_S)^2. \]
    Summing the inequalities for $x+y$ and $x-y$:
    \begin{align*}
        \norm{x+y}_{\JT}^2 + \norm{x-y}_{\JT}^2 &\ge \sum_{S \in \mathcal{P}} \left( (x_S + y_S)^2 + (x_S - y_S)^2 \right) \\
        &= \sum_{S \in \mathcal{P}} (2x_S^2 + 2y_S^2) \\
        &= 2\sum_{S \in \mathcal{P}} x_S^2 + 2\sum_{S \in \mathcal{P}} y_S^2 \\
        &= 2\norm{x}_{\JT}^2 + 2\norm{y}_{\mathcal{P}}^2.
    \end{align*}
    By assumption, $\norm{x+y}_{\JT} = \norm{x-y}_{\JT} = \norm{x}_{\JT}$, so the left-hand side is $2\norm{x}_{\JT}^2$. This simplifies to:
    \[ 2\norm{x}_{\JT}^2 \ge 2\norm{x}_{\JT}^2 + 2\norm{y}_{\mathcal{P}}^2 \implies 0 \ge 2\norm{y}_{\mathcal{P}}^2. \]
    Since the norm is non-negative, we must have $\norm{y}_{\mathcal{P}} = 0$.
\end{proof}

\begin{corollary}\label{C2.10} For $x\in \JT$ non extreme point the vector $y$ satisfying  the properties $(1), (2)$  in the previous proposition also satisfies  
$\supp(y) \subset \ran(x)$.
\end{corollary}
\begin{proof} Assume, on the contrary, that there exists $\alpha_0 \in \supp(y) \setminus \ran(x) $ and set $I_{\alpha_0} =
\{\alpha_0\}$. Choose $\mathcal{P}$  an $x$-norming partition consisting of segments contained in $\ran(x)$ and define  
$\mathcal{P}_1 = \mathcal{P} \cup \{ I_{\alpha_0} \}$. Then $\mathcal{P}_1$ remains an $x$-norming partition and violates Property $(2)$ of the previous proposition.
\end{proof}
\begin{remark}
    The strength of the above result is indicated by the following proposition, which provides a concise proof for the characterization of extreme points in the classical James space $\J$. This result was originally proved in \cite{AG2025} using a more complex argument.
\end{remark}

\begin{proposition}
    Let $x \in \J$. Then $x$ is an extreme point if and only if it is separated.
\end{proposition}

\begin{proof}
    The ``only if'' direction follows  from Proposition \ref{prop:not_sep_not_ext} stated below (non-separated implies non-extreme).
    
    Assume now that $x$ is separated but not extreme. Then there exists $y \in \J$ satisfying the conclusion of Proposition \ref{prop.1}. In particular, $y \neq 0$ and
    \[ \sum_{k \in S} y(k) = 0 \]
    for any $S \in \mathcal{P}$, where $\mathcal{P}$ is any $x$-norming partition.
    
    Set $k_0 = \min \supp(y)$. From Corollary \ref{C2.10} we have  $\supp(y) \subset \ran(x)$. Since $x$ is separated, there exists an $x$-norming partition $\mathcal{P}$ separating $k_0$ and $k_0 + 1$.
    
    Let $S \in \mathcal{P}$ be the segment such that $k_0 \in S$. Since the partition separates $k_0$ and $k_0+1$, $k_0$ is the maximum element of $S$. Consequently, $y|_S$ is different from zero only on $k_0$ (since $k_0 = \min \supp(y)$). Hence,
    \[ \sum_{k \in S} y(k) = y(k_0) \neq 0. \]
    This contradicts the property that the sum of $y$ must vanish on all segments of an $x$-norming partition, completing the proof.
\end{proof}

\section{Necessary Conditions for Extremality}

We first establish that the property of being ``separated'' is a necessary condition for extremality in $\JT$.

\begin{proposition}\label{prop:not_sep_not_ext}
    Let $x \in \JT$. If $x$ is not separated, then $x$ is not an extreme point.
\end{proposition}

\begin{proof}
    Suppose $x$ is not separated. By definition, there exists a pair of distinct nodes $u, v \in \ran(x)$ such that for every $x$-norming partition $\mathcal{P}$, $u$ and $v$ belong to the same segment $S \in \mathcal{P}$.
    
    Let $\mathbb{P}_{\text{sep}}$ denote the set of all admissible partitions $\mathcal{P}$ in which $u$ and $v$ belong to different segments. Notice that always could assume that $u,v$ belong that some element of the $x$-norming partition. Indeed if one or both do not belong that means that they belong to $\ran(x) \setminus \supp(x)$ and we can add the singleton(s) in the partition. As in the proof of Proposition \ref{P1}(i) we consider  
    \[ \mathbb{P}_{\text{sep}}\subset \Pi  \quad \text{and it is  easy to check that  it is closed.} \]
    Since the function $f_x :  \mathbb{P}_{\text{sep}} \to \mathbb{R} $ is continuous (Prop. \ref{P1}(i)) we conclude that  there exists $\delta > 0$ such that:
    \[ \sup_{\mathcal{P} \in \mathbb{P}_{\text{sep}}} \norm{x}_{\mathcal{P}}^2 = \norm{x}_{\JT}^2 - \delta. \]
    
    Let $\epsilon > 0$ be sufficiently small such that $4\epsilon \norm{x} + 2\epsilon^2 < \delta$. Define the perturbation $y = \epsilon(e_u - e_v)$.
    
    We claim that $\norm{x \pm y}_{\JT} = \norm{x}_{\JT}$. Indeed, for any partition $\mathcal{P}$:
    \begin{enumerate}
        \item If $\mathcal{P}$ puts $u, v$ in the same segment $S$, then $\sum_{S} y = \epsilon - \epsilon = 0$. Thus $\norm{x \pm y}_{\mathcal{P}} = \norm{x}_{\mathcal{P}}$. Since the supremum of $\norm{x}_{\mathcal{P}}$ is attained on such partitions, $\norm{x \pm y}_{\JT} \ge \norm{x}_{\JT}$.
        \item If $\mathcal{P}$ separates $u, v$, then:
        \begin{align*}
            \norm{x \pm y}_{\mathcal{P}}^2 &\le \norm{x}_{\mathcal{P}}^2 + 4\epsilon\norm{x}_{\mathcal{P}} + 2\epsilon^2 \\
            &\le (\norm{x}_{\JT}^2 - \delta) + 4\epsilon\norm{x}_{\JT} + 2\epsilon^2.
        \end{align*}
        By the choice of $\epsilon$, this quantity is strictly less than $\norm{x}_{\JT}^2$.
    \end{enumerate}
    Thus, $\norm{x+y}_{\JT} = \norm{x-y}_{\JT} = \norm{x}_{\JT}$, and since $y \neq 0$, $x$ is not an extreme point.
\end{proof}
\begin{example}[Separated Support does not imply Separated Range]
    Consider the vector $x \in \JT$ defined by:
    \[ x = e_{\emptyset} + e_{00} + e_{01} \]
    
    \textbf{1. Norm and Partitions:}
    The support is $\supp(x) = \{ \emptyset, 00, 01 \}$, and the range is $\ran(x) = \{ \emptyset, 0, 00, 01 \}$.
    To compute the norm, the root $\emptyset$ must merge with one of its descendants to maximize the square sum. There are two symmetric optimal partitions:
    \begin{itemize}
        \item $\mathcal{P}_1 = \{ \{\emptyset, 0, 00\}, \{01\} \}$. 
        Sum of first segment: $1+0+1=2$. Total square: $2^2 + 1^2 = 5$.
        \item $\mathcal{P}_2 = \{ \{\emptyset, 0, 01\}, \{00\} \}$. 
        Sum of first segment: $1+0+1=2$. Total square: $2^2 + 1^2 = 5$.
    \end{itemize}
    Thus, $\norm{x}_{\JT} = \sqrt{5}$.
    
    \textbf{2. The Support is Separated:}
    We verify that every pair in $\supp(x)$ can be separated:
    \begin{itemize}
        \item Pair $(\emptyset, 00)$: Separated by $\mathcal{P}_2$ (where $\emptyset$ merges with $01$).
        \item Pair $(\emptyset, 01)$: Separated by $\mathcal{P}_1$ (where $\emptyset$ merges with $00$).
        \item Pair $(00, 01)$: Separated in both partitions (they are incomparable).
    \end{itemize}
    
    \textbf{3. The Range is NOT Separated:}
    Consider the pair $(\emptyset, 0) \subset \ran(x)$. 
    In $\mathcal{P}_1$, they belong to the segment $\{\emptyset, 0, 00\}$.
    In $\mathcal{P}_2$, they belong to the segment $\{\emptyset, 0, 01\}$.
    Any partition that separates $\emptyset$ from $0$ would require isolating $\emptyset$ (sum 1) and treating the subtrees at $00$ and $01$ separately (sums 1 and 1), yielding a total square sum of $1^2 + 1^2 + 1^2 = 3 < 5$.
    Therefore, $\emptyset$ and $0$ belong to the same segment in \textbf{every} $x$-norming partition.
    
    \textbf{Conclusion:}
    Although $x$ is separated on its support, it is not separated on its range. Consequently, $x$ is not an extreme point. Indeed, the perturbation $y = \epsilon(e_{\emptyset} - e_0)$, for a small $\epsilon > 0$, satisfies $\norm{x \pm y}_{\JT} = \norm{x}_{\JT}$ because the sum of $y$ vanishes on any segment containing both $\emptyset$ and $0$.
\end{example}
\begin{remark}
    As observed by one of the reviewers, an interesting consequence of the tree geometry in the previous example is the exact bound on the perturbation parameter $\epsilon$. For the vector $x = e_{\emptyset} + e_{00} + e_{01}$ and the perturbation $y = \epsilon(e_{\emptyset} - e_0)$, the condition $\norm{x \pm y}_{\JT} = \norm{x}_{\JT}$ holds if and only if $|\epsilon| \le \phi^{-1}$, where $\phi$ is the golden ratio. 
    
    Indeed, while the optimal merged partitions yield exactly $\norm{x \pm y}_{\mathcal{P}}^2 = 5$, the partition consisting entirely of singletons must also not exceed $\norm{x}^2 = 5$. This requirement gives the inequality:
    \[ (1 \pm \epsilon)^2 + (\mp \epsilon)^2 + 1^2 + 1^2 \le 5 \]
    which simplifies to $\epsilon^2 \pm \epsilon - 1 \le 0$. The positive root of this quadratic equation yields the precise bound $|\epsilon| \le \frac{\sqrt{5}-1}{2} = \phi^{-1}$.
\end{remark}

\section{Sufficient Conditions for Extremality}

\subsection{$\ell_2$-Behavior and Extremality}

We begin by observing that vectors whose $\JT$-norm coincides with their $\ell_2$-norm are always extreme points. This provides a fundamental class of extreme points.

\begin{proposition}\label{prop:l2_extreme}
    Let $x \in \JT$. If $\norm{x}_{\JT} = \norm{x}_2$, then $x$ is an extreme point of the ball of radius $\norm{x}_{\JT}$.
\end{proposition}

\begin{proof}
    Let $y \in \JT$ be such that for every segment $s \in \mathcal{P}$, where $\mathcal{P}$ is an $x$-norming partition, it satisfies $\sum_{\alpha \in s} y(\alpha) = 0$. From the assumption, the partition consisting of singletons of $\ran(x)$ 
     is an $x$-norming partition. Applying the condition to singleton segments $\{ \alpha \}$, we obtain $y(\alpha) = 0$ for all $\alpha \in \ran(x)$.
    Corollary \ref{C2.10} yields that  $\supp(y) \subset \ran(x)$, so we conclude that $y = 0$. Proposition \ref{prop.1} yields that $x$ is an extreme point.
\end{proof}

\subsection{Connection to the James Space $\J$}
The local structure of $\JT$ along branches is isometric to the classical James space $\J$. Recent results by Argyros and González provide a complete characterization of extremality in $\J$ which had  been proved by S. Bellenot \cite{B} for a different norm of $\J$.

\begin{theorem}[Argyros \& González, 2025 \cite{AG2025}]\label{thm:James_space}
    Let $x \in \J$. The following are equivalent:
    \begin{enumerate}
        \item $x$ is an extreme point of the unit ball of $\J$.
        \item $x$ is separated and $\norm{x}_{\J}=1$.
        \item $\norm{x}_{\J} = \norm{x}_2 = 1$.
    \end{enumerate}
\end{theorem}

This theorem suggests that for ``branch-like'' structures in $\JT$, extremality is tightly coupled with the $\ell_2$-norm condition.

\subsection{Separated Vectors on Special Supports}

We now examine vectors supported on specific substructures of the tree. We show that for these cases, the separated property forces the norm to be the $\ell_2$-norm, thereby implying extremality.

\begin{proposition}\label{prop:special_supports}
    Let $x \in \JT$.
    \begin{enumerate}[label=(\alph*)]
        \item If $x$ is supported on a single branch and is separated, then $\norm{x}_{\JT} = \norm{x}_2$.
        \item If $x$ is supported by a family of pairwise incomparable segments and is separated, then $\norm{x}_{\JT} = \norm{x}_2$. Consequently, $x$ is an extreme point.
    \end{enumerate}
\end{proposition}

\begin{proof}
    \textbf{(a)} Suppose $\supp(x) \subset \mathbf{b}$ for some branch $\mathbf{b}$ of $\Tree$. The space of functions supported on $\mathbf{b}$ equipped with the $\JT$ norm is isometric to the James space $\J$. Since $x$ is separated in $\JT$, its restriction to $\mathbf{b}$ is separated in $\J$. By Theorem \ref{thm:James_space} (Argyros \& González), this implies $\norm{x}_{\J} = \norm{x}_2$. Since the norms coincide on the branch, $\norm{x}_{\JT} = \norm{x}_2$. By Proposition \ref{prop:l2_extreme}, $x$ is an extreme point.

    \textbf{(b)} Let $\supp(x) \subset \bigcup_{i \in I} S_i$, where $\{S_i\}_{i \in I}$ is a family of pairwise incomparable segments (i.e., no node in $S_i$ is an ancestor or descendant of a node in $S_j$ for $i \neq j$).
    
    Any segment $S$ in an admissible partition $\mathcal{P}$ of $\Tree$ is linearly ordered. Therefore, $S$ can intersect at most one segment $S_i$ from the incomparable family (or parts of it). If $S$ intersects multiple $S_i$'s, there would be comparable nodes in distinct $S_i$'s, contradicting their incomparability.
    
    Consequently, the norm calculation decouples:
    \[ \norm{x}_{\JT}^2 = \sum_{i \in I} \norm{x|_{S_i}}_{\JT}^2. \]
    Since $x$ is separated, each restriction $x|_{S_i}$ is separated. Note that a single segment $S_i$ is a subset of a branch, so $x|_{S_i}$ can be viewed as a vector in $\J$. Applying Part (a) (or Theorem \ref{thm:James_space}) to each segment, we have $\norm{x|_{S_i}}_{\JT} = \norm{x|_{S_i}}_2$.
    
    Summing these contributions:
    \[ \norm{x}_{\JT}^2 = \sum_{i \in I} \norm{x|_{S_i}}_2^2 = \norm{x}_2^2. \]
    Since $\norm{x}_{\JT} = \norm{x}_2$, $x$ is an extreme point by Proposition \ref{prop:l2_extreme}.
\end{proof}

\begin{proposition}[Well-Founded Support]
    Let $x \in \JT$ be a separated vector such that its range, $\ran(x)$, is a well-founded subtree of $\Tree$ (i.e., it contains no infinite branches). Then $x$ is an extreme point.
\end{proposition}

\begin{proof}
    Assume, for the sake of contradiction, that $x$ is not an extreme point. By the characterization of non-extremality (Proposition \ref{prop.1}), there exists a non-zero vector $y \in \JT$ such that:
    \begin{enumerate}
        \item $\norm{x}_{\JT} = \norm{x+y}_{\JT} = \norm{x-y}_{\JT}$.
        \item For every $x$-norming partition $\mathcal{P} = \{S_i\}$, $\sum_{\gamma \in S_i} y(\gamma) = 0$ for all $i$.
    \end{enumerate}
    From Corollary \ref{C2.10} we have  $\supp(y) \subset \ran(x)$. Since $\ran(x)$ is well-founded, the subset $\supp(y)$ is also well-founded.
    
    Let $A = \supp(y)$. Since $A$ is non-empty and well-founded, it possesses maximal elements. Let $\alpha \in A$ be a maximal node (meaning no descendant of $\alpha$ is in $A$). Thus $y(\alpha) \neq 0$, but $y(\beta) = 0$ for all $\beta > \alpha$.
    
    Since $x$ is separated, there exists an $x$-norming partition $\mathcal{P}$ that separates $\alpha$ from its parent (if $\alpha \neq \emptyset$). Consequently, the segment $S \in \mathcal{P}$ containing $\alpha$ must start at $\alpha$. Because $\alpha$ is maximal in $\supp(y)$, the segment $S$ contains no other nodes from $\supp(y)$. Therefore:
    \[ \sum_{\gamma \in S} y(\gamma) = y(\alpha). \]
    By condition (2) of the non-extremality characterization, this sum must be 0. Thus $y(\alpha) = 0$, which contradicts $\alpha \in \supp(y)$.
\end{proof}

\begin{corollary}
    If $x \in \JT$ is separated and has finite support, then $x$ is an extreme point.
\end{corollary}

\begin{proposition}
    Let $x \in \JT$ be a separated vector such that $\supp(x) \subset F \cup \sigma$, where $F$ is a finite set of nodes and $\sigma$ is a branch of $\Tree$. Then $x$ is an extreme point.
\end{proposition}

\begin{proof}
    Assume, for the sake of contradiction, that $x$ is not an extreme point. Then there exists a non-zero vector $y \in \JT$ satisfying the conditions of Proposition \ref{prop.1} (specifically, the sum of $y$ vanishes on every segment of every $x$-norming partition). We have (Corollary \ref{C2.10})  $\supp(y) \subset \ran(x)$.
    
    The set $K = \ran(x) \setminus \sigma$ is finite. We first show that $y$ must vanish on $K$ by backward induction. 
    Let $C = \{ a \in K : \text{for all } b > a, y(b) = 0 \text{ or } b \in \sigma \}$. Since $K$ is finite, every element in $K$ eventually leads to $\sigma$ or terminates.
    Let $a$ be a maximal element of $\supp(y) \cap K$. Any descendant of $a$ in $\supp(y)$ must belong to $\sigma$ (since $a$ is maximal in the non-branch part). However, since $\sigma$ is a single branch, if $a \notin \sigma$, $a$ cannot be an ancestor of any node in $\sigma$. Thus, $a$ has no descendants in $\supp(y)$.
    Using the separation property, we can separate $a$ from its parent. The resulting segment $S$ starts at $a$ and contains no other elements of $\supp(y)$. Thus $\sum_S y = y(a) = 0$.
    Repeating this argument (backward induction), we conclude that $y(a) = 0$ for all $a \in K$.
    
    Thus, $\supp(y) \subset \sigma$. Since $y \neq 0$, let $u = \min \supp(y)$. Then $u \in \sigma$.
    Let $v$ be the child of $u$ that lies on $\sigma$. If $v \notin \ran(x)$, then $u$ is maximal in the range, and by isolating it, we find $y(u)=0$. Thus, assume $v \in \ran(x)$.
    Since $x$ is separated, there exists an $x$-norming partition $\mathcal{P}$ separating $u$ and $v$. Let $S \in \mathcal{P}$ be the segment containing $u$. Since $S$ does not contain $v$, and $\supp(y) \subset \sigma \cup K$ (with $y|_K=0$), the segment $S$ contains no other elements of $\supp(y)$ except possibly $u$.
    Specifically, if $S$ continues to the other child $w$ of $u$, then $w \notin \sigma$, so $w \in K$ (or is outside the support). Since $y$ vanishes on $K$, the sum on $S$ is simply $y(u)$.
    
    By Proposition \ref{prop.1}, $\sum_S y = 0$, which implies $y(u) = 0$. This contradicts $u \in \supp(y)$.
\end{proof}

\begin{remark}
    We are not able to extend the above result for finitely many branches instead of one. In particular, we are not able to prove that $x \in \JT$ separated is an extreme point if $\supp(x) \subset \sigma_1 \cup \sigma_2$ with $\sigma_1, \sigma_2$ branches of $\Tree$.
\end{remark}

\subsection{Local Isolation Implies $\ell_2$ Norm}

We prove that if every node in the support can be isolated in some optimal partition, the global norm collapses to the $\ell_2$ norm. In the sequel for $\alpha \in \Tree $ we denote by $W_{\alpha}$ the wedge of $\Tree$ rooted at $\alpha$ defined by \[ W_\alpha = \{ \beta \in \Tree: \alpha \le \beta \}. \]
For $ x\in \JT$  and $\alpha \in \Tree$ we  denote by $ x|_{W_\alpha}$ the restriction of $x$ onto 
$W_\alpha $.

\begin{proposition}
    Let $x \in \JT$ be a vector such that for every $a \in \supp(x)$, the singleton $\{a\}$ belongs to some $x$-norming partition. Then:
    \[ \norm{x}_{\JT} = \norm{x}_2. \]
\end{proposition}

\begin{proof}
    Let $\mathcal{B} = \supp(x)$ be the subtree defined by the $\supp(x)$. Since $x \in \JT$, $\mathcal{B}$ is a countable set of nodes. We define the relative depth of a node $a \in \mathcal{B}$, denoted $|a|_{\mathcal{B}}$, as follows:
    \begin{itemize}
        \item If $a$ is minimal in $\mathcal{B}$ (i.e., no ancestor of $a$ is in $\mathcal{B}$), then $|a|_{\mathcal{B}}= 1$.
        \item Generally, $|a|_{\mathcal{B}} = k$ if the longest chain of ancestors of $a$ within $\mathcal{B}$ has length $k$.
        \item In the rest of the proof for $a\in \mathcal{B}$ by  the children of $a$ we mean the children with respect to the subtree $\mathcal{B}$.
    \end{itemize}
    
    We prove by induction that for every $n \in \mathbb{N}$:
    \begin{equation}\label{eq:inductive_step}
        \norm{x}_{\JT}^2 = \sum_{\{a \in \mathcal{B} : |a|_{\mathcal{B}} \le n\}} x(a)^2 + \sum_{\{ c \in \mathcal{B}: |c|_{\mathcal{B}} = n+1\}} \norm{x|_{W_c}}_{\JT}^2.
    \end{equation}
    (Note: The second sum runs over nodes $c$ that are ``roots'' of the remainder of the support at level $n+1$).

    \textbf{Base Case ($n=1$):}
    Let $M = \{ a \in \mathcal{B} : a \text{ is minimal in } \mathcal{B}\}$. These are the nodes with $|a|_{\mathcal{B}} = 1$.
    Since the subtrees rooted at distinct $a \in M$ are disjoint, we have:
    \[ \norm{x}_{\JT}^2 = \sum_{a \in M} \norm{x|_{W_a}}_{\JT}^2. \]
    Fix $a \in M$. By hypothesis, there exists an $x$-norming partition $\mathcal{P}_a$ such that $\{a\} \in \mathcal{P}_a$.
    Since $\{a\}$ is a segment in $\mathcal{P}_a$, the partition splits the subtree $W_a$ into the singleton $\{a\}$ and partitions of the subtrees rooted at the children of $a$ with respect to the subtree $ \mathcal{B} $.
    Thus, the norm squared decomposes:
    \[ \norm{x|_{W_a}}_{\JT}^2 = x(a)^2 + \sum_{c \in \operatorname{children}(a) } \norm{x|_{W_c}}_{\JT}^2. \]
    Summing over all $a \in M$:
    \[ \norm{x}_{\JT}^2 = \sum_{a \in M} \left( x(a)^2 + \sum_{c \in \operatorname{children}(a)} \norm{x|_{W_c}}_{\JT}^2 \right) = \sum_{\{a : |a|_{\mathcal{B}}=1\}} x(a)^2 + \sum_{\{c : |c|_{\mathcal{B}}=2\}} \norm{x|_{W_c}}_{\JT}^2. \]
    This establishes the formula for $n=1$.

    \textbf{Inductive Step:}
    Assume the formula holds for $n = k$.
    \[ \norm{x}_{\JT}^2 = \sum_{\{a : |a|_{\mathcal{B}} \le k\}} x(a)^2 + \sum_{\{c : |c|_{\mathcal{B}} = k+1\}} \norm{x|_{W_c}}_{\JT}^2. \]
    Consider a specific node $c$ with $|c|_{\mathcal{B}} = k+1$. By the global hypothesis, there exists an $x$-norming partition $\mathcal{P}_c$ (for the whole vector) containing $\{c\}$. The restriction of this partition to $W_c$ is necessarily an $x|_{W_c}$-norming partition.
    Therefore, just as in the base case, the existence of a norming partition isolating $c$ implies:
    \[ \norm{x|_{W_c}}_{\JT}^2 = x(c)^2 + \sum_{d \in \operatorname{children}(c)} \norm{x|_{W_d}}_{\JT}^2. \]
    Substituting this equality for every term in the second sum of the inductive hypothesis:
    \begin{align*}
        \norm{x}_{\JT}^2 &= \sum_{\{a : |a|_{\mathcal{B}} \le k\}} x(a)^2 + \sum_{\{c : |c|_{\mathcal{B}} = k+1\}} \left( x(c)^2 + \sum_{d \in \operatorname{children}(c)} \norm{x|_{W_d}}_{\JT}^2 \right) \\
        &= \left( \sum_{\{a : |a|_{\mathcal{B}} \le k\}} x(a)^2 + \sum_{\{c : |c|_{\mathcal{B}} = k+1\}} x(c)^2 \right) + \sum_{\{d : |d|_{\mathcal{B}} = k+2\}} \norm{x|_{W_d}}_{\JT}^2 \\
        &= \sum_{\{a : |a|_{\mathcal{B}} \le k+1\}} x(a)^2 + \sum_{\{d : |d|_{\mathcal{B}} = k+2\}} \norm{x|_{W_d}}_{\JT}^2.
    \end{align*}
    This completes the inductive step.

    \textbf{Conclusion:}
    Since $\norm{x}_{\JT} \le \norm{x}_2$ is generally false but $\norm{x}_{\JT} \ge \norm{x}_2$ is always true (using the partition of singletons), we check the reverse inequality. 
    The above equality derived in the limit is:
    \[ \lim_{N \to \infty} \left( \norm{x}_{\JT}^2 - \sum_{\{a : |a|_{\mathcal{B}} \le N\}} x(a)^2 \right) = \lim_{N \to \infty} \sum_{\{c : |c|_{\mathcal{B}} = N+1\}} \norm{x|_{W_c}}_{\JT}^2. \]
  It is well known that for $F\subset \Tree$ consisting of pairwise incomparable nodes we have 
  \[ \norm{x|_{ \bigcup_{ c\in F} W_c}}_{\JT} ^2 = \sum_{c\in F} \norm{ x|_{W_c}}_{\JT}^2  \quad \text{and} \quad  \sum_{c\in F_n} \norm{ x|_{W_c}}_{\JT}^2 \to 0               \]
  when $ \min \{ |c|_\Tree : c \in F_n\} \to \infty $.
    Thus, \[\norm{x}_{\JT}^2 = \sum_{a \in \mathcal{B}} x(a)^2 = \norm{x}_2^2. \]
\end{proof}

\section{Positive Extreme Points}

In this section, we focus on positive vectors ($x \ge 0$). The absence of sign cancellations allows for a constructive approach to determining the norm using a greedy algorithm. We begin by establishing the necessary notation.

\subsection{Notation and Basic Properties}

\begin{definition}[Notation]\label{No1}
    Let $x \in \JT$ and $a \in \Tree$.
    \begin{enumerate}
        \item Let $W_a = \{ b \in \Tree : a \le b \}$ denote the wedge rooted at $a$.
        \item Let $x_a$ denote the restriction of $x$ to $W_a$, i.e., $x_a = x \cdot \chi_{W_a}$.
        \item For a segment $s \subset \Tree$ we define the sum of $x$ along $s$ as:
        \[ S_s(x) = \sum_{b \in s} x(b). \]
        \item We define the maximal segment sum starting at $a$ as:
        \[ S_a(x) = \sup \{ S_s(x) : s \text{ is a segment with } \min s = a \}. \]
        \item For $x\in \JT$ we set $\mathcal{B} = \supp(x)$ the subtree defined by $x$. Let $a,c \in \mathcal{B}$ such that $c$ is a child of $a$ with respect to the subtree $\mathcal{B}$. Then we denote by $s_{a, c}$ the segment defined as 
        $\min s_{a, c} = a$ and $\max s_{a, c} = d$ where $d$ is the parent of $c$ in $\Tree$.
    \end{enumerate}
\end{definition}

\begin{remark}
  (1)   If $x$ is a positive vector in $\JT$ with finite support, then for every $a \in \supp(x)$, the supremum $S_a(x)$ is attained by some segment $s$.  For $x$ with infinite support, we observe that  if $(b_n)_n $  is a sequence of maximal segments in $W_a$
  with $  b_n  \xrightarrow{\text{pw}} b$ then $S_{b_n} (x) \to S_b(x) $.  Since  $W_a$  is compact the supremum is attained.
  
  (2)  If $x\in \JT$ is a positive vector and $a, c \in \supp(x) $ with $c$ being a child of $a$ with respect to the subtree $\supp(x)$, 
  then $S_{s_{a,c}}(x) = x(a)$ and if $s$ is a segment with $\min s= a$ and $c\in s$ then $s = s_{a,c} \cup s' $ with  $s'$ a segment and $\min s' = c$.  
\end{remark}
  
  \begin{definition}
    Let $x$ be a positive vector in $\JT$. A \textbf{canonical $x$-norming partition} $\mathcal{P}$ is an $x$-norming partition $\mathcal{P} = \{s_a: a\in F\}$ satisfying the following properties:
  \begin{enumerate}
  \item Each $s_a$ is a final segment of the tree $\Tree$.
  \item $a = \min s_a \in \supp(x)$.
  \end{enumerate}
  \end{definition}

  \begin{lemma}
    Let $x\in \JT$ be a positive vector and $\mathcal{P}$ an $x$-norming partition. Then there exists a 
  $\mathcal{P}_1 = \{ s_a: a \in F\}$ canonical $x$-norming partition such that for every $s \in \mathcal{P}$ with
    $s \cap \supp(x) \ne \emptyset$ there exists an $s_a \in \mathcal{P}_1$ such that 
         \[ s(x) = \sum_{c\in s} x(c) = \sum_{c\in s_a} x(c) = s_a(x). \]
  \end{lemma}      
    \begin{proof}
      We first observe that for every $s\in \mathcal{P}$ such that $s \cap \supp(x) \ne \emptyset$, if $s$ is finite then $\max s$ is a maximal node of the subtree $\mathcal{B}= \supp(x)$.
    
    We define the canonical $x$-norming partition $\mathcal{P}_1$ as follows.
    \begin{enumerate}
    \item If $s\in \mathcal{P} \: \text{with} \: s\cap \supp(x) \ne \emptyset $ is finite we set
     \[ a =\min s\cap \supp(x) \quad \text{and} \quad  \beta = \max s\cap \supp(x). \] 
     In this case we define $s_a$ as any final segment of $\Tree$ such that $\min s_a = a $ and $\beta \in s_a$.
    \item If $s\in \mathcal{P} \: \text{with} \: s\cap \supp(x) \ne \emptyset $ is infinite then clearly it is a final segment of $\Tree$ and setting $a = \min s\cap \supp(x)$ we define $s_a = s\cap W_a$.
    \end{enumerate}
     It is easy to check that $\mathcal{P}_1 = \{ s_a : a \in F\}$ is a canonical $x$-norming partition satisfying the desired property.
    \end{proof}

  \begin{remark}
    As follows from the preceding result, considering only canonical $x$-norming partitions we do not restrict the variety of $x$-norming partitions concerning their action on the positive vector $x$. In the following two subsections of this section we will consider exclusively canonical $x$-norming partitions.
  \end{remark}

\subsection{The Greedy Algorithm and Canonical Norming Partitions}

In the context of the Greedy Algorithm, when we refer to the "children" of a node $a \in \supp(x)$, we refer to the children within the subtree defined by the support of $x$.

The following lemma provides the recursive step for calculating the norm of a positive vector. It essentially states that if we know how to optimally partition the subtrees, the optimal strategy at the root $a$ is to merge with the child offering the largest segment sum.

\begin{lemma}[Recursive Norm Formula]\label{lemma:greedy_step}
    Let $x \in \JT$ be a positive vector and $a \in \supp(x)$ not a maximal node with respect to $\supp(x)$. Let $C(a)$ denote the set of children of $a$ with respect to the subtree $\mathcal{B} = \supp(x) $. Assume that for every child $c \in C(a)$ and for every final segment $s_c \subset W_c $ such that
     $S_{s_c}(x) = S_c(x)$, there exists a canonical $x_c$-norming partition $\mathcal{P}_c$ that includes the segment $s_c$.
    
    Then the norm of the restriction $x_a$ satisfies:
    \begin{enumerate}
    \item
    \begin{equation}
        \norm{x_a}_{\JT}^2 = x(a)^2 + 2x(a)S_{c_0}(x) + \sum_{c \in C(a)} \norm{x_c}_{\JT}^2
    \end{equation}
    where $c_0 \in C(a)$ is a child that maximizes the segment sum, i.e., $S_{c_0}(x) = \max_{c \in C(a)} S_c(x)$.
    \item For every $s$ final segment such that $a = \min s$ and $S_a = S_s$, there exists a canonical $x_a$-norming partition
    $\mathcal{P}$ such that $s \in \mathcal{P}$.
    \end{enumerate}
\end{lemma}

\begin{proof}
  (1) Let $\mathcal{P}$ be a canonical $x_a$-norming partition of $\supp(x_a)$. The node $a$ must belong to some segment $s \in \mathcal{P}$. Since $a$ is the root of $W_a$, $s$ must start at $a$.
    
    If $s = \{a\}$, then the partition splits $a$ from all its descendants. The norm would be:
    \[  \norm{x_a}_{\JT}^2 = x(a)^2 + \sum_{c \in C(a)} \norm{x_c}_{\JT}^2. \]
    However, since $x \ge 0$, merging $a$ with a segment from a child is always superior or equal to isolating $a$.
    
    Suppose $s$ continues into the wedge $W_{c'}$ for some child $c'$. Then $s = s_{a, c'} \cup s'$, where $s_{a, c'}$ is as in Definition \ref{No1} and $s'$ is a final segment starting at $c'$. The contribution of this segment is:
    \[ \left( x(a) + S_{s'}(x) \right)^2 = x(a)^2 + 2x(a)S_{s'}(x) + S_{s'}(x)^2. \]
    Then 
    \[ \norm{x_a}_{\JT}^2 = x(a)^2 + 2x(a)S_{s'}(x) + \norm{x_{c'}}_{\mathcal{P}_{s'}}^2 + \sum_{c \neq c'} \norm{x_c}_{\JT}^2 \]
    where $\mathcal{P}_{s'} = \{ s \in \mathcal{P}: s \subset W_{c'} \} \cup \{ s' \}$.
    
    To maximize the total norm $\norm{x_a}^2$, it is enough, if it is possible, to choose $s'$ such that 
    \begin{enumerate}
    \item $S_{s'} = S_{c'}$
    \item $\norm{x_{c'}}_{\mathcal{P}_{s'} } = \norm{x_{c'}}_{\JT}$
    \end{enumerate}
    This is indeed possible because of our inductive assumptions. Hence 
      \begin{equation}
        \norm{x_a}_{\JT}^2 = x(a)^2 + 2x(a)S_{c_0}(x) + \sum_{c \in C(a)} \norm{x_c}_{\JT}^2
    \end{equation}

    (2) We write $s = s_{a, c_0} \cup s'$ and we set $c_0 = \min s'$. Clearly $c_0 \in C(a)$ and $S_{s'}(x_{c_0}) = 
    S_{c_0}(x_{c_0})$. By our assumptions there exists a canonical $x_{c_0}$-norming partition $\mathcal{P}_{c_0}$ with
    $s'\in \mathcal{P}_{c_0}$. Further for every $c\in C(a)$, $c\ne c_0$, we choose a canonical $x_c$-norming partition $\mathcal{P}_c$. Then the partition 
    \[ \mathcal{P} = \{ s\} \cup (\mathcal{P}_{c_0}\setminus \{s'\}) \cup \bigcup_{c\ne c_0}\mathcal{P}_c \]
    satisfies 
       \begin{equation}
        \norm{x_a}_{\mathcal{P}}^2 = x(a)^2 + 2x(a)S_{s'}(x) + \sum_{c \in C(a)} \norm{x_c}_{\JT}^2
    \end{equation}
    Since $S_{s'}(x_{c_0}) =  S_{c_0}(x_{c_0})$, the first part of the lemma yields that 
    \[  \norm{x_a}_{\mathcal{P}}^2  =  \norm{x_a}_{\JT}^2. \]
    Hence $\mathcal{P}$ is an $x_a$-norming partition.
    
    \end{proof}
    The following is a consequence of the previous result and it is proved by induction.
    \begin{proposition}\label{P5.5} Let $x\in \JT$ with finite support and unique minimal element $a$ of $\mathcal{B}= \supp(x)$. Then:
    \begin{enumerate}
    \item For every canonical $x$-norming partition $\mathcal{P}$, the unique $s\in \mathcal{P}$ such that $a\in s$ satisfies the following:
    \begin{enumerate}
   \item $ S_s(x) = S_a(x)$ 
   \item If $s= s_{a, c'} \cup s'$  with $\min s' = c'$, then we have
        \begin{equation}
        \norm{x}_{\JT}^2 = x(a)^2 + 2x(a)S_{s'}(x) + \sum_{c \in C(a)} \norm{x_c}_{\JT}^2
    \end{equation}
    where $ C(a) = \text{children}(a)$ with respect to the subtree $\mathcal{B}$.
    \end{enumerate}
    
    \item Conversely, for every segment $s$ with $a\in s$ and $S_s(x) = S_a(x)$, there exists a canonical $x$-norming partition 
    $\mathcal{P}$ with $s\in \mathcal{P}$.
   
    \end{enumerate}
    \end{proposition}
    \begin{proof} 
    \begin{enumerate}
    \item We prove by induction on the height $h(\mathcal{B})$ of the tree $\mathcal{B} = \supp(x)$, simultaneously (1) and (2).  
    We recall that 
    
    \[ h(\mathcal{B})= \max\{|a|_\mathcal{B} : a\in \mathcal{B} \}. \]
    
    For $k=1$ the result is trivial. Assume that (1) and (2) are true for all $x\in \JT$ with finite support and $h(\supp(x)) = k$,
    then for $x\in \JT$ with finite support and $h(\supp(x)) = k+1$, Lemma \ref{lemma:greedy_step} yields the result ending the proof.

    \item It is already proved.
        
    \end{enumerate}
\end{proof}
\begin{definition}[Greedy Algorithm]
    For a positive vector $x \in \JT$, the \textbf{Greedy Algorithm} constructs inductively a partition $\mathcal{P}_G$ as follows:
    We set $\mathcal{B} = \supp(x)$ and for $u \in \mathcal{B}$ we set $|u|_\mathcal{B}$ the height of $u$ as an element of the subtree $\mathcal{B}$. The partition $\mathcal{P}_G$ consists of final segments $s$ with $\min s \in \mathcal{B}$.

    \begin{enumerate}
        \item For every minimal node of $\mathcal{B}$ (i.e. $|u|_\mathcal{B}= 1 $), we choose a final segment $s$ with 
        $\min s = u $ and $S_s(x) = S_u(x)$.
        \item Let $k\in \N$ and  $\{ I_a \}_{a\in F_k}$  have been chosen such that 
        \begin{enumerate}
        \item  $I_a $ is a final segment and $\min I_a = a $.
        \item 
        \[ F_k \subset \{ a\in \mathcal{B} : |a|_\mathcal{B} \le k \} \subset \bigcup_{a \in F_k } I_a \]
        \end{enumerate}
      \item In the inductive step we set
       \[ L_{k+1}= \{ a\in \mathcal{B}: |a|_\mathcal{B} = k+1 \} \quad \text{and } \quad M_{k+1} = L_{k+1} \setminus
       \bigcup_{a\in F_k} I_a  \]
       For $a \in M_{k+1} $, we choose a final segment $I_a$ such that
        \[ \min I_a = a \quad \text{ and} \quad S_{I_a}(x_a) = S_a(x_a). \]
       This completes the inductive step.
        \item This process defines a partition $\mathcal{P}_G$ of pairwise disjoint final segments covering the support of the vector $x$.
    \end{enumerate}
\end{definition}
\begin{remark} \label{R5.6}(1) It is clear that for a given  positive $x\in \JT$, the Greedy Algorithm may produce different $\mathcal{P}_G$ partitions since for some $a \in \supp(x)$ we may have multiple choices for the segment $I_a$.

(2) It follows from the inductive definition that for $x\in \JT$ and $ n\in \N$ 
\[ \left\{ a\in \supp(x): |a|_\mathcal{B} \le n \right\} \subset \bigcup \left\{ I_a : |a|_\mathcal{B} \le n\right\} \]

(3) If $a\in \supp(x)$ is such that $I_a \in \mathcal{P}_G$, it is easy to check that the partition 
  $$\mathcal{Q}_a = \{ I_\beta \in \mathcal{P}_G: \beta \in \supp(x_a) \} $$ is an outcome of applying the Greedy Algorithm to the vector $x_a$.
 
 (4) More generally, if $u\in \supp(x)\cap I_a$ for some $I_a\in \mathcal{P}_G$, then the set 
  \[ \mathcal{Q}_u = \{ I_\beta \in \mathcal{P}_G : I_\beta \subset W_u  \} \cup \{I_a \cap W_u\} \]
  is the outcome of Greedy Algorithm applied on $x_u$.

 \end{remark}
 
\begin{proposition}\label{prop.5.6}
    Let $x \in \JT$ be a positive vector with finite support. Then the Greedy Algorithm produces a canonical $x$-norming partition.
\end{proposition}

\begin{proof} For $x\in \JT $ positive with finite support, we set $\mathcal{B} = \supp(x) $ the finite subtree of $\Tree$.
We proceed by induction on $h(\mathcal{B})$. If $h(\mathcal{B}) =1 $, then $\supp(x)$ consists of finitely many maximal
elements and clearly the partition of the singletons in $\supp(x)$ defines both the unique outcome of the Greedy Algorithm 
applied to $x$ as well as the unique canonical $x$-norming partition. Assume now that the conclusion is true for all positive $x$ with $h(\mathcal{B})\le k$, and let $x$ be positive with $h(\mathcal{B})=k+1$. For simplicity, assume also that $\mathcal{B}$ has a unique minimal point. The general case follows by applying the following proof to each minimal element of $\mathcal{B}$.

Let $\mathcal{P}_G$ be the outcome of Greedy algorithm applied to $x$ and $a$ the minimal element of $\mathcal{B}$.
Let $s_a \in \mathcal{P}_G $ with $\min s_a = a$. Clearly $s_a$ is a final segment. Set $c'$ to be the child of $a$ in $\mathcal{B}$ such that $c' = \min( s_a \cap W_{c'})$, and also set $s' = s_a\cap W_{c'}$.

From Remark \ref{R5.6} (3) we have that for every   $ c \in C(a) = \operatorname{children}(a)$ in $\mathcal{B}$ 
 $c\ne c'$, the family 
$$\mathcal{P}_c = \left \{ s_\beta : s_\beta \subset W_c \right \} $$
is the outcome of the Greedy Algorithm applied to $x_c$. Hence, by induction, it defines a canonical $x_c$ partition.

Also, from Remark \ref{R5.6} (4), the family 
  $$ \mathcal{P}_{c'} = \left \{ s_\beta : s_\beta \subset W_{c'} \right \} $$
 is the outcome of the Greedy Algorithm applied to $x_{c'}$ and, again by inductive assumption, is a canonical $x_{c'}$-norming partition.
 
 Applying the partition $\mathcal{P}_G$ to $x$, we obtain
        \begin{equation} \label{eq7}
         \norm{x}_{\mathcal{P}_G }^2 = x(a)^2 + 2x(a)S_{s'}(x) + \sum_{c \in C(a)} \norm{x_c}_{\JT}^2
    \end{equation}
From the fundamental properties of the Greedy Algorithm, we have that
  $$  S_{s'} (x) = \max \left\{ S_c(x) : c\in C(a) \right \} $$
 Therefore $S_{s'}(x)$ maximizes \eqref{eq7} for the various partitions; hence
$$ \norm{x}_{\JT} = \norm{x}_{\mathcal{P}_G}. $$
This ends the proof.
\end{proof}

\subsection{The General Case: Infinite Support}

We extend the validity of the Greedy Algorithm to all positive vectors $x \in \JT$ with infinite support. The strategy relies on approximating $x$ by vectors supported on subtrees consisting of finitely many branches, where the algorithm is shown to be norming.

\subsubsection{Approximation by Finitely Branched Subtrees}

\begin{notation}\label{n5}Let $x \in \JT$ be a positive vector. Let $\mathcal{S} = \{s_\alpha\}_{\alpha \in F}$ be the family of the final segments generated by the Greedy Algorithm applied to $x$, where $F \subset \supp(x)$ is the set of starting nodes of these segments, or it is a canonical $x$-norming partition with the same properties.

For each $n \in \N$, we define the truncated set of starting nodes:
\[ F_n = \{ \alpha \in F : |\alpha| \le n \}. \]
We define the approximating domain $Q_n$ as:
\[ Q_n = T_n \cup \bigcup_{\alpha \in F_n} s_\alpha, \]
where $T_n = \{ u \in \Tree : |u| \le n \}$. Note that $Q_n$ is a \textbf{complete subtree} of $\Tree$.
We define the vector $y_n$ as the restriction of $x$ to $Q_n$:
\[ y_n = x \cdot \chi_{Q_n}. \]
Notice that $\supp(y_n) \subset \bigcup_{\alpha \in F_n} s_{\alpha}$.

We define in the same manner  $Q_n$, $y_n$ when we start with $\mathcal{S} = \{s_\alpha\}_{\alpha \in F}$ as a canonical $x$-norming partition.
\end{notation}

\begin{proposition}\label{P5.9} Let $x\in \JT$ and $\mathcal{P} = \{ I_a\}_{a\in F}$ be a family of pairwise disjoint final segments of $\Tree$ with $F\subset \supp(x)$ and $a = \min I_a$. The following are equivalent:
\begin{enumerate}
\item The family $\mathcal{P}$ is the outcome of the Greedy Algorithm.
\item The family $\mathcal{P}$ covers $\supp(x)$ and for every $a \in F$ we have $S_a(x) = S_{I_a}(x).$
\end{enumerate}
\end{proposition}

\begin{proof} (1) $\Rightarrow$ (2). It is clear from the definition of the Greedy Algorithm.

(2) $\Rightarrow$ (1). We set $\mathcal{B} = \supp(x)$ and using induction we prove that for $k\in \N$,
the family $ \mathcal{P}_k = \{ I_a : |a|_\mathcal{B} \le k\}$ is the outcome of the first $k$ steps of the Greedy Algorithm.
The proof is straightforward and it is left to the reader. This completes the proof.
\end{proof}
\begin{lemma}[Approximation and Stability] \label{lemma5.6}
    Let $x \in \JT$ be positive and $\mathcal{S} = \{ s_{\alpha} : \alpha \in F \} $ be a partition generated by the Greedy Algorithm on $x$. For $y_n$ as above, the following hold:
    \begin{enumerate}
        \item The family $\mathcal{S}_n = \{ s_\alpha : \alpha \in F_n \} $ is the outcome of the Greedy Algorithm applied to $y_n$.
        \item The sequence $(y_n)_n$ converges in norm: $\norm{x - y_n}_{\JT} \to 0$.
    \end{enumerate}
\end{lemma}

\begin{proof}
    \textbf{1. Stability of the Greedy Choice:}
    It is easy to check that the family $\mathcal{S}_n = \{ s_\alpha : \alpha \in F_n \} $ satisfies the assumptions of Proposition \ref{P5.9} (2); hence the proposition yields that it is the outcome of the Greedy Algorithm applied to $y_n$.
    
    \textbf{2. Convergence:}
    Observe that the set $R_n = Q_n \setminus T_n$ is a complete subtree of $\Tree$. Hence, setting 
    $z_n = y_n|_{R_n}$ and $x_n = x|_{\Tree \setminus T_n}$, Lemma \ref{lem1} yields 
    \[ \norm{z_n}_{\JT} \le \norm{x_n}_{\JT}. \]
    Therefore, $\norm{z_n}_{\JT} \to 0$ and $\norm{x_n}_{\JT} \to 0$.
    Since $x - y_n = x_n - z_n$, we conclude that \[\norm{x - y_n}_{\JT} \to 0. \]
\end{proof}

\subsubsection{Verification for Finitely Many Branches}\label{lemma5.7}

The vectors $y_n$ constructed above have a specific property: their support is contained in $T_n$ plus a finite collection of segments. Thus, $\supp(y_n)$ is contained in a finite union of branches.

\begin{lemma}[Finitely Many Branches Case]\label{L5.8}
    Let $z \in \JT$ be a positive vector such that $\supp(z)$ is contained in a union of finitely many branches. Then the Greedy Algorithm produces a canonical $z$-norming partition.
\end{lemma}

\begin{proof}
    Let $\mathcal{B} = \{b_1, \dots, b_k\}$ be the finite set of branches containing $\supp(z)$.
    There exists a level $N$ of $\Tree$, sufficiently large, such that for any node $u$ with $|u| \ge N$, the wedge $W_u$ intersects at most one branch from $\mathcal{B}$.
    
    For any such node $u\in \supp(z)$ (with $|u| = N$), the restriction $z_u = z|_{W_u}$ is supported on a single path (a subset of some $b_i$). For a positive vector on a line, the Greedy Algorithm trivially selects the entire path (or the maximal positive sub-segment), which corresponds to the sum along that path. This is known to be the norming partition for the restriction of a positive vector to a branch (James space case).
    
    Let $\mathcal{P}_G = \{I_a\}_{a\in F}$ be a partition generated by the Greedy Algorithm applied to $z$ with each $I_a$ a final segment of $\Tree$ and $ a = \min I_a \in \supp(z)$. It follows from the preceding paragraph that every $a\in F$ 
    satisfies $|a|_\Tree \le N$.
    
    We define a new vector $y$ as follows: 
    \[
     y(a) = 
     \begin{cases}
      z(a) \quad & \text{if} \quad  |a|_\Tree < N \\
      \sum_{a \le \beta} z(\beta) \quad & \text{if}  \quad  |a|_\Tree = N \: \text{and} \: a\in b_k \in \mathcal{B} \\
      0 \quad & \text{ otherwise}
     \end{cases}\]
     Observe that $y$ has a finite support and the above partition $\mathcal{P}_G$ is also the outcome of the Greedy Algorithm applied to $y$ (since, as we noticed before, for every $I_a\in \mathcal{P}_G$ we have $|a|_\Tree \le N$). Hence Proposition \ref{prop.5.6} derives that $\mathcal{P}_G$ is a canonical $y$-norming partition. Observe also that       \[ \norm{y}_{\mathcal{P}_G} = \norm{z}_{\mathcal{P}_G}. \]
     It remains to show that    $ \norm{y}_{\JT}= \norm{z}_{\JT}$. For this, observe that for every canonical $y$-norming partition $\mathcal{P}$, we may consider the every $ s\in\mathcal{P}$ satisfies that $ s\subset b_m$ for some $1\le m \le k$. Then we have $\norm{y}_{\mathcal{P}} = \norm{z}_{\mathcal{P}}$, hence 
        \[ \norm{y}_{\JT} \le \norm{z}_{\JT}. \]
       Consider a canonical $z$-norming partition $\mathcal{Q} = \{ s_a : a\in P\} $ with each $s_a$ a final segment of $\Tree$ and $a= \min s_a \in \supp(z)$. As in the case of $\mathcal{P}_G$, every $a\in P$ must satisfy $|a|_\Tree \le N$, hence we have that 
    \[ \norm{y}_{\mathcal{Q}} = \norm{z}_{\mathcal{Q}} \quad \text{yielding}  \quad  \norm{z}_{\JT} \le \norm{y}_{\JT}. \]
    From the above, we conclude that $\mathcal{P}_G$ is a canonical $z$-norming partition, ending the proof.

    \end{proof}
      
\subsubsection{Main Result for Positive Vectors}

We combine the approximation and the finite branch case to prove the result for general positive vectors.

\begin{theorem}\label{T5.9}
    Let $x \in \JT$ be a positive vector. The partition $\mathcal{S}$ generated by the Greedy Algorithm is a canonical $x$-norming partition.
\end{theorem}

\begin{proof}
    Let $(y_n)_n$ be the sequence of approximations defined on $Q_n$.
    By Notation \ref{n5}, $\supp(y_n)$ consists of the finite tree $T_n$ and the segments $\{s_\alpha\}_{\alpha \in F_n}$. This support is covered by finitely many branches. Let also $\mathcal{P}_G$ be the outcome of the Greedy Algorithm applied to $x$. We set $\mathcal{P}_n = \{ s_a : \alpha \in F_n\}$; from 
    Lemma \ref{lemma5.6}, $\mathcal{P}_n$ is the outcome of the Greedy Algorithm applied to $y_n$. From Lemma \ref{L5.8} we have that $\mathcal{P}_n$ is a canonical $y_n$-norming partition, hence
    \[ \norm{y_n}_{\JT} = \norm{y_n}_{\mathcal{P}_n}. \]
    
    Also, we have that
    \[ \mathcal{P}_G = \bigcup_n \mathcal{P}_n. \]
    Since $x \ge y_n \ge 0$, and the partition $\mathcal{P}_G$ is consistent with $\mathcal{P}_n$ on the support of $y_n$, we have:
    \[ \norm{x}_{\mathcal{P}_G} \ge \norm{y_n}_{\mathcal{P}_G} \ge \norm{y_n}_{\mathcal{P}_n} = \norm{y_n}_{\JT}. \]
    Taking the limit as $n \to \infty$:
    \[ \norm{x}_{\mathcal{P}_G} \ge \lim_{n \to \infty} \norm{y_n}_{\JT} = \norm{x}_{\JT}. \]
    Since $\norm{x}_{\mathcal{P}_G} \le \norm{x}_{\JT}$ by definition, we conclude $\norm{x}_{\mathcal{P}_G} = \norm{x}_{\JT}$.
    Thus, the Greedy Algorithm defines a canonical $x$-norming partition.
\end{proof}
We prove the following:
\begin{proposition}\label{Prop.5.10}
    Let $x\in \JT$ be a positive vector and $\alpha \in \Tree$ a minimal node in $\supp(x)$. Then for every segment $s$ of $\Tree$ with $\alpha \in s$ satisfying $ S_s(x) = S_\alpha(x)$, there exists a canonical $x$-norming partition $\mathcal{P}$ with $s\in \mathcal{P}$.
\end{proposition}
\begin{proof}
    Observe that for every $s$ as in the statement, there exists a $\mathcal{P}_G$ outcome of the Greedy Algorithm with $s_a = s$. By the previous theorem, $\mathcal{P}_G$ is a canonical $x$-norming partition.
\end{proof}

\subsection{Converse of the Greedy Algorithm}

We have established that the Greedy Algorithm constructs a canonical norming partition for positive vectors. We now prove the converse: for any positive vector $x$, a canonical $x$-norming partition must be the outcome of applying the Greedy Algorithm.

\begin{definition}[Maximal Segment Property]
    Let $x \in \JT$ be a positive vector. A partition $\mathcal{P}$ satisfies the \textbf{Maximal Segment Property} if every segment $s \in \mathcal{P}$ is a final segment with $a = \min s \in \supp(x)$ and we have $S_s(x_a)= S_a(x_a)$.
\end{definition}
Proposition \ref{P5.9} derives the following.
\begin{corollary}\label{C5.13}
    Let $x\in \JT$ be a positive vector. Every canonical $x$-norming partition $\mathcal{P}$ with the maximal segment property is the outcome of the Greedy Algorithm applied to $x$.
\end{corollary}
We pass now to prove the following.

\begin{proposition}\label{P5.14}
    Let $x \in \JT$ be a positive vector. Then every canonical $x$-norming partition satisfies the maximal segment property. 
\end{proposition}

\begin{proof}
    We proceed in three steps: first for finite support, then for support on finitely many branches, and finally for the general case.
    
    \textbf{Step 1: Finite Support}
    Let $x$ have finite support. We set $\mathcal{B} = \supp(x)$ and by induction on $h(\mathcal{B})$, we show that for every canonical $x$-norming partition $\mathcal{P}= \{ s_a : a\in F\}$, we have $S_a(x_a) = S_{s_a}(x_a)$. 
    
    Assume that for $x\in \JT$ with $h(\mathcal{B}) \le k$ and every canonical $x$-norming partition, the conclusion is true.
    Let $x$ have $h(\mathcal{B}) = k+1$ and let $\mathcal{P}$ be a canonical $x$-norming partition. Then for $a$ a minimal node 
    of $\mathcal{B}$, there exists a unique $s\in \mathcal{P}$ with $a= \min s$. Then Proposition \ref{P5.5} (1) yields that
    $S_a(x)=S_s(x)$. Hence the segment $s$ satisfies the desired property. Any other $t\in \mathcal{P}$ with 
    $\min t$ not minimal in $\mathcal{B}$ satisfies the desired property from the inductive assumption. This ends the proof for the case of $x$ with finite support.
    \end{proof}
    
     \textbf{Step 2: Finitely Many Branches} We prove the following.
     
     \begin{lemma} \label{L28} 
       Let $z$ be a positive vector and $\mathcal{B} = \{ b_1, \dots, b_n \} $ a finite family of branches of $\Tree$ such that $\supp(z) \subset \bigcup_{k\le n} b_k$. Then every $z$-norming partition satisfies the maximal segment property.
   \end{lemma}
   \begin{proof}
     Assume on the contrary that there exists a canonical $z$-norming partition $\mathcal{P}$ that fails the maximal segment property. Let $c\in F$ violate the maximal segment property, i.e.,
     \[ S_{s_c} (z_c ) < S_c(z_c). \]
     Then there exists a final segment $t$ with $\min t = c$ and $S_t(z_c) > S_{s_c}(z_c)$. We choose $N$ sufficiently large such that:
      \begin{enumerate}
      \item Any two $b_k, b_m$ with $k \ne m$ are separated at the level $N$.
      \item 
      \[ \text{If} \quad t_N = \{ a\in t: |a|_\Tree \le N\}, \quad \text{then} \quad \sum_{a\in t_N } z(a) > S_{s_c} (z_c). \]
      \end{enumerate}
      We define a new vector $y$ as follows:
      \[
       y(a) = 
       \begin{cases}
        z(a) \quad & \text{if} \quad  |a|_\Tree < N \\
        \sum_{a \le \beta} z(\beta) \quad & \text{if}  \quad  |a|_\Tree = N \: \text{and} \: a\in b_k \in \mathcal{B} \\
        0 \quad & \text{ otherwise}
       \end{cases}
      \]
      It is easy to check that
       $\norm{y}_\mathcal{P} = \norm{z}_\mathcal{P} = \norm{z}_{\JT}$. Also, for any $y$-norming partition $\mathcal{Q}$,with its elements subsets of $ \cup_{k\le n} b_k $ it is straightforward that $\norm{y}_\mathcal{Q} =\norm{z}_\mathcal{Q}$, hence $\norm{y}_{\JT} \le \norm{z}_{\JT}$. This yields that $\mathcal{P}$ is a canonical $y$-norming partition which violates the maximal segment property, a contradiction as $y$ has finite support.
   \end{proof}

        \textbf{Step 3: General Case}
    Let $x$ be an arbitrary positive vector and $\mathcal{P}$ be a canonical $x$-norming partition.
    Suppose $\mathcal{P}=\{ s_a: a\in F \}$ as before fails the maximal segment property. Then there exists $c\in F$ and $t$ a final segment of $\Tree$ with $\min t=c$ and $S_{s_c}(x_c) < S_t(x_c)$.
    
    Consider the approximations $y_n$ supported on the complete subtrees $Q_n$ which are contained in the union of finite branches (Notation \ref{n5}). For sufficiently large $N$, $Q_N$ contains $c$ and sufficiently long initial segments of the optimal paths starting at $c$ such that 
      \[ \text{for} \quad  t_N = \{ a\in t: |a|_\Tree \le N\}, \quad \text{then} \quad  \sum_{a\in t_N } z(a) > S_{s_c} (x_c). \]
    
    Let $\mathcal{P}_N \subset \mathcal{P}$ be defined by $\mathcal{P}_N = \{ s \in \mathcal{P} : |\min s| \le N\}$. Since $Q_N$ is a complete subtree and $\supp(y_N) \subset \bigcup_{s\in \mathcal{P}_{N} } s_a$, Lemma \ref{lem1} (2) yields that $\mathcal{P}_N$ is a canonical $y_N$-norming partition which violates the maximal segment property. This and Lemma \ref{L28} 
    yield a contradiction, completing the proof.
    
    Proposition \ref{P5.14} and Corollary \ref{C5.13} prove the next theorem:
    \begin{theorem}
      Let $x \in \JT$ be a positive vector. Then every canonical $x$-norming partition $\mathcal{P}$ is the outcome of the Greedy Algorithm applied to $x$.
    \end{theorem}
  
\subsection{Characterization of Positive Extreme Points}

We are now ready to state and prove the complete characterization of extreme points for positive vectors in $\JT$.

\begin{theorem}
    Let $x \in \JT$ be a positive vector ($x \ge 0$). Then $x$ is an extreme point if and only if $x$ is separated.
\end{theorem}

\begin{proof}
    \textbf{($\Rightarrow$) Direction:}
    We have already established in Proposition \ref{prop:not_sep_not_ext} that if a vector is not separated, it cannot be an extreme point. This holds for all vectors, including positive ones.
    
    \textbf{($\Leftarrow$) Direction:}
    Assume that $x$ is separated but, for the sake of contradiction, that $x$ is \textbf{not} an extreme point.
    
    By the characterization of non-extreme points (Proposition \ref{prop.1}), there exists a non-zero vector $y \in \JT$
    with $\supp(y) \subset \ran(x)$ such that:
    \begin{enumerate}
        \item $\norm{x}_{\JT} = \norm{x+y}_{\JT} = \norm{x-y}_{\JT}$.
        \item For every $x$-norming partition $\mathcal{P}$, and every segment $S \in \mathcal{P}$, the sum of the perturbation vanishes: $\sum_{\gamma \in S} y(\gamma) = 0$.
    \end{enumerate}
    
    Since $y \neq 0$, we can choose a node $a \in \supp(y)$. Observe that $a \in \ran(x)$.
    
    \textbf{Step 1: Isolate $a$ from above.}
    Since $x$ is separated, there exists an $x$-norming partition $\mathcal{P}_1$ that separates $a$ from its parent (if $a \neq \emptyset$). Let $s_a \in \mathcal{P}_1$ be the segment containing $a$. Because $a$ is separated from its parent, $a$ must be the minimal element (the starting node) of $s_a$.
    By the property of $y$, we must have:
    \begin{equation}\label{eq:sum_sa}
        \sum_{\gamma \in s_a} y(\gamma) = 0.
    \end{equation}
    
    If $s_a = \{a\}$, then $y(a) = 0$, which contradicts $a \in \supp(y)$. Thus, $s_a$ must contain at least one descendant of $a$ with respect to $\supp(x)$. Let $\beta$ be the child of $a$ contained in $s_a$. We can write $s_a = s_{a, \beta} \cup t_\beta$, where $t_\beta$ is a segment starting at $\beta$.
    Equation (\ref{eq:sum_sa}) becomes:
    \[ y(a) + \sum_{\gamma \in t_\beta} y(\gamma) = 0 \implies \sum_{\gamma \in t_\beta} y(\gamma) = -y(a) \neq 0. \]

    \textbf{Step 2: Isolate $a$ from below.}
    Since $x$ is separated, there exists another $x$-norming partition $\mathcal{P}_2$ that separates $a$ from $\beta$.
    In this partition, $a$ belongs to some segment $s' \in \mathcal{P}_2$ which does \textbf{not} contain $\beta$.
    This implies that the restriction of $\mathcal{P}_2$ to the subtree $W_\beta$ (rooted at $\beta$) is an $x_\beta$-norming partition. Let us denote the square value of the optimal partition on $W_\beta$ as $\norm{x_\beta}_{\JT}^2$.
    
    \textbf{Step 3: The Contradiction.}
    We return to the segment $s_a = s_{a, \beta} \cup t_\beta$ from Step 1.
    Since $\mathcal{P}_1$ is an $x$-norming partition, it must be consistent with the Greedy Algorithm. This implies that the tail $t_\beta$ is an optimal path starting at $\beta$. Specifically, $t_\beta$ must be part of some $x_\beta$-norming partition for the subtree $W_\beta$ (Theorem \ref{T5.9}).
    
    Let $\mathcal{P}_\beta$ be an $x_\beta$-norming partition that contains the segment $t_\beta$. We can construct a new partition $\mathcal{P}^*$ for the whole vector $x$ by combining:
    \begin{itemize}
        \item The partition $\mathcal{P}_\beta$ for the subtree $W_\beta$ (containing $t_\beta$).
        \item All $s \in \mathcal{P}_2$ such that $s \cap W_\beta = \emptyset$.
    \end{itemize}
    This partition $\mathcal{P}^*$ is $x$-norming.
    
    However, strictly speaking, we rely on the property of non-extreme points: \textbf{Any} $x$-norming partition must have $y$-sum zero on its segments.
    
    The segment $t_\beta$ is a segment in the valid $x$-norming partition $\mathcal{P}^*$.
    Thus, we must have:
    \[ \sum_{\gamma \in t_\beta} y(\gamma) = 0. \]
    
    But from Step 1, we established that:
    \[ \sum_{\gamma \in t_\beta} y(\gamma) = -y(a) \neq 0. \]
    
    This is a contradiction. Therefore, the assumption that $x$ is not an extreme point must be false.
\end{proof}

\subsection{Equal Sums Property for Complete Positive Extreme Points}

We prove the Equal Sums Property for complete positive, separated vectors in $\JT$, namely for positive separated $x\in \JT$ such that $\supp(x) = \ran(x)$. Note that in this subsection, when we discuss children of a node, we refer to the children in the full dyadic tree $\Tree$.

\begin{definition}[Equal Sums Property]
    A vector $x \in \JT$ satisfies the \textbf{Equal Sums Property} if for any node $u \in \supp(x)$ and any two final segments $S_1, S_2$ with $\min S_i = u$ for $i=1,2$, we have:
    \[ \sum_{\gamma \in S_1} x(\gamma) = \sum_{\delta \in S_2} x(\delta). \]
\end{definition}
\begin{definition}[Complete Vectors]
  A vector $x\in \JT$ is called \textbf{complete} if its support is a complete subtree of 
$\Tree$. This is equivalent to the property $\supp(x) = \ran(x)$.
\end{definition}

\subsubsection{Local Equality}

\begin{lemma}\label{lemma5.15}
    Let $x \in \JT$ be positive and separated. Let $a \in \supp(x)$. If $a$ is not a maximal node of the $\supp(x)$, and $c_0, c_1$ are its children in $\Tree$, then $W_{c_i} \cap \supp(x) \neq \emptyset$ for $i=0,1$, and:
    \[ S_{c_0}(x) = S_{c_1}(x), \]
    where $S_u(x)$ denotes the maximal segment sum starting at $u$.
\end{lemma}

\begin{proof}
    If $a$ is not maximal, it has descendants in $\supp(x)$. Suppose, for contradiction, that $W_{c_1} \cap \supp(x) = \emptyset$. Then all descendants of $a$ lie in $W_{c_0}$. In this case, the Greedy Algorithm at $a$ would always choose $c_0$ (since the sum at $c_1$ is 0). This effectively implies $a$ and $c_0$ are never separated, contradicting the separated property of $x$. Thus, both subtrees must contain support.
    
    Furthermore, if $S_{c_0}(x) \neq S_{c_1}(x)$, say $S_{c_0} > S_{c_1}$, the greedy choice would be unique and strictly prefer $c_0$. Again, this would imply $a$ and $c_0$ belong to the same segment in every optimal partition, violating the separated condition. Thus $S_{c_0} = S_{c_1}$.
\end{proof}

\subsubsection{Global Constancy via Density}

We shall prove the result for complete $x\in \JT$ with a unique minimal element of the support of $x$. The general case is proved by applying the proof in each minimal element of the $\supp(x)$. Let 
$\Branches$ denote the set of branches of the tree $\Tree$ intersecting $\supp(x)$. Note that every $s\in \Branches$ contains the minimum of the $\supp(x)$. We define the maximal branch sum functional:
\[ \Sigma(x) = \max \{ S_s(x) : s \in \Branches \}. \]

\begin{proposition}
    Let $x \in \JT$ be a complete positive, separated vector with a unique minimal  element of $\supp(x)$. Then for every branch $s$ with $s\cap \supp(x) \ne \emptyset$ the sum of $x$ along the branch $s$   is constant. That is,  $$S_s(x) = \Sigma(x)$$
\end{proposition}

\begin{proof}
    We prove this by showing that the set of branches attaining the maximum sum is dense in $\Branches$.
    
    We claim that for every node $a \in \supp(x)$, there exists a branch $s$ containing $a$ such that $S_s(x) = \Sigma(x)$.
    We proceed by induction on the structure of the support (or implicitly, the depth relative to the support).
    
    If $a$ is a maximal node of $\supp(x)$, any branch passing through $a$ collects the full sum up to $a$ and 0 thereafter. Since $x$ is complete and separated, if $\beta$ is the parent of $a$, then $\beta \in \supp(x)$; hence, the inductive assumption and Lemma \ref{lemma5.15} yield that $\Sigma(x) = S_b(x)$ where $b$ is any branch containing $a$.

    If $a$ is not maximal, then its parent $\beta \in \supp(x)$, hence by Lemma \ref{lemma5.15} and the inductive assumption, there exists a branch $b$ with $a\in b$ and $\Sigma(x) = S_b(x)$. 
    
    This completes the proof. Indeed, if 
    $s\in \Branches$ is a branch with $s \cap \supp(x) \ne \emptyset$, then either $\#(s\cap \supp(x)) < \infty$,
    and in this case $s$ contains a maximal element of the support of $x$, hence $\Sigma(x) = S_s(x)$. 
    The second alternative is that $s\cap \supp(x)$ is an infinite set. In this case, for every $a\in s\cap \supp(x)$, from the second part of the above proof, there exists a branch $b_a$ such that $a\in b_a$ and $\Sigma(x) = S_{b_a}(x)$, which shows that $s$ is approximated by branches $b$ with $\Sigma(x) = S_b(x)$.

    Thus, the set of branches $\{ s \in \Branches : S_s(x) = \Sigma(x) \}$ is dense in the set of all branches intersecting $\supp(x)$. Since the mapping $s \mapsto S_s(x)$ is continuous on the branch space $\Branches$ (for $x \in \JT$), and the function is constant on a dense subset, it must be constant on the whole set.
\end{proof}

As a consequence, we have the following:

\begin{theorem}\label{T5.18}
    Every $x \in \JT$ complete positive and separated satisfies the Equal Sums Property.
\end{theorem}
\begin{remark} 
  (1) As pointed out by the AI tool ChatGPT, this result does not hold for arbitrary positive extreme points. Indeed, the positive vector
  \[ x = e_\emptyset + 2e_{00} + e_{10} + 2e_{11} \]
  is separated, hence extreme, but fails the equal sums property.

  (2) There is a wider class of positive extreme points including the complete ones for which Theorem \ref{T5.18} remains valid. For those positive extreme points $x$, their support satisfies the following property: for every $a\in \supp(x)$,
  either $a$ is a maximal element of the $\supp(x)$ or it has exactly two children in $\supp(x)$. It is easy to check that any positive extreme point such that its support satisfies this property also satisfies the equal sums property.
\end{remark}

\section{Topological Properties of the Set of Extreme Points}

We address the question of whether the set of extreme points is closed in the norm topology. We show that this is true for the set of separated points. Hence, if the equivalence between extremality and the separated property holds, then this set is indeed closed.

\begin{proposition}
    The set of separated vectors of the unit ball of $\JT$ is closed in the norm topology.
\end{proposition}

\begin{proof}
    We show that the property of being separated is preserved under norm convergence.
    
    Let $(x_n)_n$ be a sequence of separated vectors in the unit sphere such that $\norm{x_n - x}_{\JT} \to 0$. We must show that $x$ is separated.
    
    Let $u, v \in \ran(x)$ be distinct nodes. We claim that there exists an $x$-norming partition separating $u$ and $v$.
    There exists an $n_0 \in \N$ such that for $n > n_0$, $u$ and $v$ belong to $\ran(x_n)$. For each $n > n_0$, choose a 
    $\mathcal{P}_n$ $x_n$-norming partition that separates $u, v$. Since $\norm{x_n - x}_{\JT} \to 0$, we have that 
    $\norm{x}_{\mathcal{P}_n} \to \norm{x}_{\JT}$.
  
  Choose a subsequence $\mathcal{P}_{n_k} \xrightarrow{\text{pw}} \mathcal{P}$. Then Proposition \ref{P1} and Remark 
  \ref{Re1} derive that $\mathcal{P}$ is an $x$-norming partition which separates the elements $u, v$.
\end{proof}
 
\section*{Acknowledgments}
The author extends his thanks to Pavlos Motakis and the reviewers of the paper for their comments and remarks that helped improve the content and the presentation of the research.

\section*{Declaration of Generative AI and AI-assisted technologies in the writing process}
It is worth pointing out that this research was initiated by an attempt to evaluate the mathematical capabilities of the AI model Gemini 2.5. Specifically, the author asked the model to check if the vector $x = \sum _{|\alpha |\le 3 }e_{\alpha}$ is an extreme point of the ball of radius $\norm{x}$. After a few seconds, it provided a proof that led to Proposition \ref{prop.1}, which easily yields that separated vectors with finite support are indeed extreme. Beyond this, the author had fruitful discussions with Gemini 3 regarding the mathematical formulations. During the preparation of this work, the author also used Gemini 3 to format the LaTeX file and improve the readability of the text. After using this tool, the author reviewed and edited the content as needed and takes full responsibility for the final content of the publication.

\end{document}